\input xy
\xyoption{all}\xyoption{poly}
\newdir{ >}{{}*!/-9pt/@{>}}

\documentclass[reqno]{amsart}
\usepackage{amsmath,latexsym,amssymb,mathrsfs}

\usepackage{textcomp}
\usepackage{graphicx}

\usepackage{hyperref}
\usepackage[all]{xy}

\newcommand{\Qnd}{\mathsf{Qnd}}
\newcommand{\Ker}{\mathsf{Ker}}
\newcommand{\Qndt}{\mathsf{Qnd}^{\ast}}
\newcommand{\Inn}{\mathsf{Inn}}
\newcommand{\Aut}{\mathsf{Aut}}
\newcommand{\Eq}{\mathsf{Eq}}

\theoremstyle{plain}
\newtheorem{theorem}{Theorem}[section]
\newtheorem{lemma}[theorem]{Lemma}
\newtheorem{proposition}[theorem]{Proposition}
\newtheorem{corollary}[theorem]{Corollary}

\theoremstyle{definition}
\newtheorem{definition}[theorem]{Definition}

\theoremstyle{remark}
\newtheorem{remark}[theorem]{Remark}

\usepackage{graphicx}
 \usepackage{color}

\newdir{ >}{{}*!/-10pt/\dir{>}}
\def\cartesien{%
    \ar@{-}[]+R+<6pt,-1pt>;[]+RD+<6pt,-6pt>%
    \ar@{-}[]+D+<1pt,-6pt>;[]+RD+<6pt,-6pt>%
  }
\begin{document}

\title{ON FACTORISATION SYSTEMS FOR SURJECTIVE QUANDLE HOMOMORPHISMS}

\author{VALERIAN EVEN and MARINO GRAN }

\address{Institut de Recherche en Math\'ematique et Physique, \\  Universit\'e catholique de Louvain, \\ Chemin du Cyclotron 2, 1348 Louvain-la-Neuve, Belgium.}

\begin{abstract}
We study and compare two factorisation systems for surjective homomorphisms in the category of quandles. The first one is induced by the adjunction between quandles and trivial quandles, and a precise description of the two classes of morphisms of this factorisation system is given. In doing this we observe that a special class of congruences in the category of quandles always permute in the sense of the composition of relations, a fact that opens the way to some new universal algebraic investigations in the category of quandles.
The second factorisation system is the one discovered by E. Bunch, P. Lofgren, A. Rapp and D. N. Yetter. We conclude with an example showing a difference between these factorisation systems. \\

\end{abstract}

\maketitle
\section*{Introduction}

A \emph{quandle} is a set $A$ equipped with two binary operations $\lhd$ and $\lhd^{-1}$ such that the following identities hold:
\begin{itemize}
\item[(A1)] $a \lhd a = a = a \lhd^{-1}a$ for all $a \in A$ (idempotency);
\item[(A2)] $(a \lhd b) \lhd^{-1} b = a = (a \lhd^{-1} b) \lhd b$ for all $a,\ b \in A$ (right invertibility);
\item[(A3)] $(a \lhd b) \lhd c = (a \lhd c) \lhd (b \lhd c)$ and $(a \lhd^{-1} b) \lhd^{-1} c = (a \lhd^{-1} c) \lhd^{-1} (b \lhd^{-1} c)$ for all $a,\ b,\ c \in A$ (self-distributivity).
\end{itemize}
This structure was introduced independently in the 1980's by S.~V.~Matveev \cite{M} and by D.~Joyce \cite{J} who first named such a structure a quandle. One of the first goals of this structure was to encode properties of group conjugation in order to find a suitable tool to study objects like the Wirtinger presentation of a knot group. The knot quandle is an invariant allowing one to distinguish two knots up to orientation (see the survey \cite{Carter} for an introduction to this topic, for instance).

A map $f \colon A \rightarrow B$ from a quandle $A$ to another quandle $B$ is a quandle homomorphism when  it preserves the operations: $f(a \lhd a') = f(a) \lhd f(a')$ \linebreak and $f(a \lhd^{-1}a') = f(a) \lhd^{-1} f(a')$.
 Quandles and quandle homomorphisms form a category, denoted by $\Qnd$, which is actually a \emph{variety} in the sense of universal algebra, namely a class of algebras determined by some (finitary) operations satisfying suitable identities, with morphisms preserving these operations \cite{BS}. Of course, in the case of $\Qnd$, there are just two binary operations $\lhd$ and~$\lhd^{-1}$, which are required to satisfy the identities (A1), (A2) and (A3) above. The category $\Qnd$ contains the subcategory $\Qndt$ whose objects are the \emph{trivial quandles}: these are the quandles satisfying the additional identity $a \lhd b =a$. Of course, any set can be uniquely endowed with a trivial quandle structure.
 
 There is an inclusion functor $U \colon \Qndt \rightarrow \Qnd$, which simply ``forgets'' the fact that a quandle is trivial. This functor has a left adjoint 
  $\pi_0 \colon \Qnd \rightarrow \Qndt$ associating with any quandle its set of connected components (seen as a trivial quandle):
   \begin{equation}\label{adj}
\xymatrix{\Qnd \ar@/^1pc/[rr]^{\pi_0} & \perp & \ar@/^1pc/[ll]^{U}\Qndt.} \tag{A}
\end{equation}
In order to explain the definition of the functor $\pi_0 \colon \colon \Qnd \rightarrow \Qndt$, let us now recall some elementary facts about quandles.
  The identities (A2) and (A3) in the definition of a quandle $A$ imply that the right actions, denoted by~$\rho_b \colon A \rightarrow~A$ and defined by~$\rho_b(a) = a \lhd b$ for all $a \in A$, are automorphisms (=bijective quandle homomorphisms). We write~$\Inn(A)$ for the subgroup of the group $\Aut(A)$ of automorphisms of $A$ generated by all such $\rho_b$, with~$b \in A$. $\Inn(A)$ is called the subgroup of \emph{inner automorphisms} of $A$. 
A quandle $A$ is \emph{connected} if $\Inn(A)$ acts transitively on $A$. A \emph{connected component} of $A$ is an orbit under the action of $\Inn(A)$. Two elements~$a$ and~$b$ of $A$ are in the same orbit if one can find a chain of elements~$a_i \in A$, for~$1~\leq~i~\leq~n$, linking the elements~$a$ and $b$
\[a \lhd^{\alpha_1} a_1 \lhd^{\alpha_2} a_2 \dots \lhd^{\alpha_n} a_n = b,\] where, by convention, one writes
\[a \lhd^{\alpha_1} a_1 \dots \lhd^{\alpha_n} a_n := (\dots ((a \lhd^{\alpha_1} a_1) \lhd^{\alpha_2} a_2) \dots )\lhd^{\alpha_n} a_n\] with $\lhd^{\alpha_i} \in \{\lhd, \ \lhd^{-1}\}$ for all $1 \leq i \leq n$.

The functor $\pi_0 \colon \Qnd \rightarrow \Qndt$ sends a quandle $A$ to its set $\pi_0(A)$ of connected components, this latter being seen as a trivial quandle.
This functor $\pi_0$ is left adjoint of the inclusion functor $U$, and the $A$-component of the unit $\eta_A \colon A \rightarrow U \pi_0 (A)$ of the adjunction is simply given by the canonical projection sending an element of the quandle to its connected component.

The reflector $\pi_0  \colon \Qnd \rightarrow \Qndt$ does not preserve \emph{all} pullbacks. Nevertheless, this functor still has some nice left exactness properties \cite{E}, in the sense that it preserves a suitable class of pullbacks (see Theorem \ref{Admissible}).
This fact implies the existence of a canonical factorisation system $(\mathcal{E},\mathcal{M})$ for surjective homomorphisms of quandles associated with the adjunction (\ref{adj}). We describe this factorisation system in Section~\ref{factsurj}, by using an important property of permutability of a class of congruences in $\Qnd$ (Lemma \ref{permut}),  explicitly described in Section \ref{locapermut}, that is of independent interest. This factorisation system $(\mathcal{E},\mathcal{M})$ satisfies a characteristic property of the so-called \emph{reflective} ones \cite{CHK}: $\mathcal E$ is the class of surjective homomorphisms which are inverted (= sent to an isomorphism) by the reflector $\pi_0 \colon \Qnd \rightarrow \Qndt$, and  for two composable surjective homomorphisms $f$ and $g$, then~$g \in \mathcal{E}$ whenever $f \circ g \in \mathcal{E}$ and~$f \in \mathcal{E}$. The class $\mathcal M$ is the class of trivial extensions (also called trivial coverings) in the sense of categorical Galois theory~\cite{BJ,JK} (see also~\cite{E1,E2,E}).


In the last section we turn our attention to another factorisation system for surjective homomorphisms which was considered by E. Bunch, P. Lofgren, A. Rapp, and D. N. Yetter in \cite{BLRY}. We conclude the article with an example showing that this latter factorisation system, unlike the previous one, does not satisfy the typical property of reflective factorisation systems.

\section{Congruences and local permutability}\label{locapermut}
In this first section we prove that there is a special class of congruences in the category $\Qnd$ of quandles which \emph{permute} in the sense of the composition of relations, with any other congruence on the same quandle. Such congruences naturally arise from the adjunction \eqref{adj} as kernel congruences of the components of the unit of the adjunction, as we explain below, and they have been considered in \cite{BLRY}, for a different purpose.

Adopting the usual terminology from universal algebra we call an equivalence relation $R \subset A \times A$ on (the underlying set of) a quandle $A$, a \emph{congruence}, if it has the additional property that $R$ is also a subquandle of the product quandle~$(A~\times~A, \lhd, \lhd^{-1})$, so that for any $(a,b) \in R$ and $(a',b') \in R$ both the elements $$(a,b) \lhd (a',b') = (a \lhd a', b\lhd b')$$ and $$(a,b) \lhd^{-1} (a',b') = (a \lhd^{-1} a', b\lhd^{-1} b')$$
belong to the relation $R$.

Let us now recall how to associate a congruence with any subgroup of the group~$\Inn(A)$ of inner automorphisms of a quandle $A$:
\begin{definition}
If $N$ is a subgroup of $\Inn(A)$, one defines an equivalence relation~$\sim_N$ on $A$ by setting: $a \sim_N b$ if and only if $a$ and $b$ lie in the same orbit via the induced action of $N$ on $A$.
\end{definition}
As shown in \cite{BLRY} (Theorem $6.1$), the equivalence relation~$\sim_N$ is a \emph{congruence} on~$A$ if and only if $N$ is a \emph{normal} subgroup of $\Inn(A)$. 

We are now going to show that this kind of congruences always \emph{permute}, in the sense of the composition of relations, with any other congruence in the variety $\Qnd$. 
\begin{definition}\label{relational}
Given two congruences $R$ and $S$ on a quandle $A$, their (relational) composite $S \circ R$ is defined as the following relation on $A$:
$$S \circ R = \{ (a,b) \in A \times A \, \mid \exists c \in A \  \mathrm{with}  \, (a,c) \in R\, \, \mathrm{and } \, \, (c, b) \in S \}.$$
\end{definition}
The subset $S \circ R$ of $A \times A$ is always a subquandle of $A \times A$, and it is reflexive and symmetric. When the congruences $R$ and $S$ permute, i.e. $S \circ R= R \circ S$, then~$S \circ R$ is also transitive, and it is a congruence on~$A$.
In general, however, there is no reason for two congruences $R$ and $S$ on the same quandle $A$ to permute in $\Qnd$.

\begin{lemma}\label{permut}
Let $A$ be a quandle, $R$ a congruence on $A$, and 
$N$ a normal subgroup of $\Inn(A)$. Then the congruences $R$ and $\sim_N$ permute:
$$\sim_N \circ R = R \circ \sim_N.$$\end{lemma}

\begin{proof}
Let $(a,b)\in\ \sim_N \circ R$, so that there exists $c \in A$ such that $(a,c)\in R$ \linebreak and $(c,b) \in\ \sim_N$. In particular, there is an automorphism $n \in N$ such that its action~$c^n$ on the element $c$ is~$b$, i.e. $c^n=b$. It follows that $(a,b) \in R \circ \sim_N$, since $(a, a^n) \in\ \sim_N$ and~$(a^n,b)= (a^n, c^n) \in R$. Accordingly, one has the inclusion~$\sim_N \circ R \subset R \circ \sim_N$, and then the equality $$\sim_N \circ R = R \circ \sim_N.$$
\end{proof}

\begin{remark}
It is well known that, for a general variety $\mathbb V$ of algebras, the permutability of the composition of congruences on any algebra is equivalent to the existence of a ternary term $p(x,y,z)$ satisfying the identities $p(x,y,y)=x$ \linebreak and $p(x,x,y)=y$. This is a classical theorem due to A.I. Mal'tsev (see \cite{BS}, Theorem $12.2$, for example, for a proof). When this is the case one says that $\mathbb V$ is a Mal'tsev variety \cite{S}. Any variety of algebras whose theory contains the operations and identities of the theory of groups is a Mal'tsev variety: it suffices to choose for~$p$ the term $p(x,y,z)= x \cdot y^{-1} \cdot z$. The variety $\Qnd$ of quandles is \emph{not} a Mal'tsev variety, and for this reason we find it interesting to observe that the special type of congruences described in the Lemma \ref{permut} still permute in $\Qnd$.
\end{remark}
Given a surjective homomorphism $f \colon A \to B$ in $\Qnd$, we write $\Eq(f)$ for its \emph{kernel congruence}  $\Eq(f)= \{ (a,b) \in A \times A \, \mid \, f(a) =f(b) \}$. 

When $\Eq(f) = \sim_N$ for a normal subgroup $N$ of $\Inn(A)$, one always has that~$\sim_N = \sim_{\Ker(\Inn(f))}$ (see Theorem $7.1$ in \cite{BLRY}).
This observation suggests to consider the following class of morphisms:
\begin{definition}
 $$\mathcal{E}_1 = \{f \colon A \rightarrow B \, \mathrm{in} \, \Qnd \, \vert \,  f \,\mathrm{ is \, a \, surjective \, homomorphism \, and \, }  \Eq(f) = \sim_{\Ker(\Inn(f))} \,   \}$$ 
\end{definition}
Thanks to Lemma \ref{permut} we know that the kernel congruences of the arrows in the class $\mathcal{E}_1$ have the strong property that they permute with any other congruence. \\

\begin{remark}\label{permutabilityofeta}
Any kernel congruence $\sim_{\Inn(A)} $ of the $A$-component $\eta_A \colon A \rightarrow U\pi_0 (A)$ of the unit $\eta$ of the adjunction between $\Qnd$ and $\Qndt$ belongs to the class ${\mathcal E}_1$. From now on, we shall drop the full inclusion $U \colon \Qndt \rightarrow \Qnd$ from the notations. For instance, we shall write $\eta_A \colon A \rightarrow \pi_0(A)$ for the $A$-component of the unit of the adjunction.
\end{remark}

In order to prove a remarkable property of a special type of pushouts in the category $\Qnd$ we need to fix some more notation. Given a homomorphism~$f \colon A \rightarrow~B$ one writes $f$ for the relation 
 $$
 
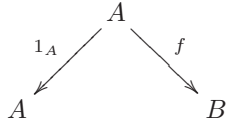
\xymatrix{
&  A \ar[dl]_{1_A} \ar[dr]^f &\\
A& & B
 }
 $$
 representing its graph: $f = \{(a,f(a))\mid\ a\in A\}$. The opposite relation
  $$
 
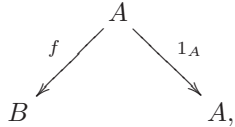
\xymatrix{
&  A \ar[dl]_{f} \ar[dr]^{1_A}  &\\
B& & A,
 }
 $$
denoted by $f^o$, is given by $f^o = \{(f(a),a) \mid a\in A\}$.

Given a relation $R$ on $A$ and a homomorphism~$f \colon A \to B$, the direct image $$f(R) = \{ (f(a),f(a')) \, \mid \, (a,a')\in R \} \subset B \times B$$ of $R$ by $f$ is also given by the relational composite $$f(R) = f \circ R \circ f^o,$$ 
where the composition of relations is the obvious extension of Definition \ref{relational} to general relations:
\begin{eqnarray}
  f \circ R \circ f^o & = & \{ (x,y) \in B \times B \, \mid \,\exists a\in A,  \exists a' \in A \, \mathrm{\, with} \,  (x,a) \in f^o, (a, a') \in R\, ,  \, (a', y) \in f\} \nonumber \\
   &=&\{ (x,y) = (f(a), f(a')) \in B \times B \, \mid \, (a,a') \in R \} \nonumber \\
  &=& f(R). \nonumber 
  \end{eqnarray}

A homomorphism $f \colon A \rightarrow B$ is surjective if and only if~$f~\circ~f^o =~\Delta_B$, where $$\Delta_B = \{ (b,b) \, \mid \, b \in B\}$$ is the equality relation on $B$. The kernel congruence $\Eq (f)$ of a homomorphism $f$ can be written as the composite $f^o \circ f$. 

 Lemma \ref{permut} implies a useful property of a special type of pushouts in $\Qnd$ (see \cite{CKP} for a general study of permutability of equivalence relations in regular categories). 
\begin{lemma}\label{specialPushouts}
Let 
$$ \xymatrix{A \ar[r]^f \ar[d]_g & B \ar[d]^{\overline{g}}  \\
C \ar[r]_{\overline{f}} & D
}$$
be a pushout of surjective homomorphisms in $\Qnd$ with the property that $f \in \mathcal{E}_1$. Then the canonical factorisation $(g,f) \colon A \rightarrow C \times_D B$ to the pullback of $\overline{f}$ and $\overline{g}$ is a surjective homomorphism.
\end{lemma}
\begin{proof}
The fact that $f \in \mathcal{E}_1$ implies that $$\Eq(f) \circ \Eq(g)= \Eq(g) \circ \Eq(f) = \Eq(f) \vee \Eq(g)$$ is the supremum of $\Eq(f)$ and $\Eq(g)$ as congruences on $A$ \cite{CKP}. Moreover, the fact that the square is a pushout implies that 
$\Eq(t) = \Eq(f) \vee \Eq(g)$, with $t = \overline{f} \circ g$.
Consequently, \begin{eqnarray}\label{Equat} t ^o \circ t =  f^o \circ f \circ g^o \circ g.\end{eqnarray}
The direct image of $(g,f) \colon A \rightarrow C \times_D B$ is given by the relation $f  \circ g^o$, 
whereas the relation $(C \times_D B, \pi_1, \pi_2)$ given by the pullback projections  is ${\overline{g}}^o \circ \overline{f}$. Finally, by composing \eqref{Equat} on the left by $f$ and on the right by $g^o$ one obtains the equality
\begin{eqnarray}
f \circ t ^o \circ t \circ g^o &= & f \circ  f^o \circ f \circ g^o \circ g \circ g^o \nonumber 
\end{eqnarray}
so that \begin{eqnarray}
f \circ f^o \circ \overline{g}^o \circ \overline{f} \circ g \circ g^o &= &  f  \circ g^o   \nonumber 
\end{eqnarray}
(since $f \circ f^o=\Delta_B$ and $g \circ g^o =  \Delta_C$), and then
\begin{eqnarray}
{\overline{g}}^o \circ \overline{f} &= & f  \circ g^o, \nonumber 
\end{eqnarray}
as desired.
\end{proof}
In particular, the following useful result holds:
\begin{corollary}\label{QuandleAdj}
For any surjective homomorphism $f \colon A \rightarrow B$ in $\Qnd$ the commutative square 
\begin{equation}\vcenter{\label{canonical}
\xymatrix{A \ar[r]^f \ar[d]_{\eta_A} & B \ar[d]^{\eta_B} \\
 \pi_0(A) \ar[r]_{\pi_0(f)} &  \pi_0 (B)
}}
 \end{equation}
 where $\eta$ is the unit of the adjunction \eqref{adj} 
 has the property that the canonical arrow~$( \eta_A,f) \colon A \rightarrow  \pi_0(A) {\times}_{\pi_0 (B)} B$ to the pullback (of $\pi_0 (f)$ and $\eta_B$) is surjective.
\end{corollary}
\begin{proof}
This follows immediately from Lemma \ref{specialPushouts}, from Remark \ref{permutabilityofeta} and the fact that the square \eqref{canonical} is a pushout. This latter property follows immediately from the fact that, given a surjective homomorphism in $\Qnd$ whose domain is a trivial quandle, its codomain is also trivial (in other words $\Qndt$ is closed in $\Qnd$ under quotients).
\end{proof}
The following property will also be needed:
\begin{corollary}\label{inducedissurjective}
Given a surjective quandle homomorphism $f \colon A \rightarrow B$, the induced homomorphism $\overline{f} \colon \sim_{ \mathsf{Inn}(A)} \rightarrow  \sim_{\mathsf{Inn}(B)}$ making the diagram
\begin{equation}\vcenter{\label{induced}
\xymatrix{\sim_{\mathsf{Inn}(A)} \ar@{.>}[r]^{\overline{f}}  \ar@<2pt>[d] \ar@<-2pt>[d]  & \sim_{\mathsf{Inn}(B)} \ar@<2pt>[d] \ar@<-2pt>[d]  \\
A \ar[r]_f & B. }}
\end{equation}
commute is surjective.
\end{corollary}
\begin{proof}
The commutative diagram \eqref{induced} is obtained 
by taking the kernel congruences $\sim_{\mathsf{Inn}(A)}$ and $\sim_{\mathsf{Inn}(B)}$ of $\eta_A$ and $\eta_B$ in diagram \eqref{canonical}, respectively.
Let us write~$\gamma = (\eta_A, f)  \colon A \rightarrow \pi_0 (A) \times_{\pi_0 (B)} B $ for the induced homomorphism to the pullback $(\pi_0 (A) \times_{\pi_0 (B)} B,p_1,p_2)$ of $\pi_0(f)$ and $\eta_B$ such that $p_1 \circ \gamma = \eta_A$ and $p_2 \circ \gamma = f$.
To see that the induced homomorphism $\overline{f}$ is surjective, it will suffice then to check that the direct image $\gamma (\sim_{\mathsf{Inn}(A)})$ of $ \sim_{\mathsf{Inn}(A)}$ along $\gamma$ is $\Eq(p_1)$, since this will imply that $$f( \sim_{\mathsf{Inn}(A)} ) = (p_2 \circ \gamma) (\sim_{\mathsf{Inn}(A)}) = p_2 (\gamma (\sim_{\mathsf{Inn}(A)})) = p_2 (\Eq(p_1) )= \sim_{\mathsf{Inn}(B).}$$
The equality $\gamma (\sim_{\mathsf{Inn}(A)})= \Eq(p_1)$ follows from the fact that $\gamma \circ \gamma^o = \Delta_{\pi_0(A) {\times}_{\pi_0 (B)} B}$ (since $\gamma$ is a surjective homomorphism by Corollary \eqref{QuandleAdj}):
$$ \gamma (\sim_{\mathsf{Inn}(A)})= \gamma \circ {\eta_A}^o \circ \eta_A \circ \gamma^o=  \gamma \circ \gamma^o \circ p_1^o \circ p_1 \circ \gamma \circ \gamma^o = p_1^o \circ p_1 = \Eq(p_1).$$
\end{proof}

\section{The induced factorisation system for surjective morphisms}\label{factsurj}
In this section we show that there is a factorisation system of surjective homomorphisms induced by the reflective subcategory~$\Qndt$ of $\Qnd$, and we describe it explicitly. We refer the reader to the reference \cite{CJKP} for a discussion of general factorisation systems (see also \cite{Chi} for the notion of factorisation system for a given class of morphism).\\

Let $\mathcal F$ be the class of surjective homomorphisms in $\Qnd$. 
\begin{definition}\label{system}
A pair $(\mathcal{E},\mathcal{M})$ of classes of maps in $\Qnd$ constitutes a \emph{factorisation system} for $\mathcal F$ if
\begin{itemize}
\item[(i)]  $\mathcal{E}$ and $\mathcal M$ contain the identities and are closed under composition with isomorphisms;
\item[(ii)] every map in $\mathcal{F}$ can be written as $m \circ e$ with $m\in \mathcal{M}$ and $e \in \mathcal{E}$;
\item[(iii)] given any commutative square
$$
\xymatrix{A \ar[r]^e  \ar[d]_u & B \ar[d]^{v} \\ C  \ar[r]_m  & D 
}$$
with $e$ in $\mathcal E$ and $m$ in $\mathcal M$, there is a unique arrow $w \colon B \rightarrow C$ such that $w \circ e = u$ and $m \circ w = v$.
\end{itemize}
\end{definition}

The restriction of the notion of factorisation system to the class $\mathcal F$ of surjective homomorphisms is related to the fact that the functor $\pi_0$ has a nice exactness property only with respect to the class $\mathcal F$ of surjective homomorphisms in $\Qnd$. 
This fact is explained in the following result from \cite{E}, which is now based on Corollary \eqref{QuandleAdj}:
\begin{theorem}\label{Admissible}
In the adjunction \eqref{adj}  the reflector $\pi_0 \colon \Qnd \rightarrow \Qndt$ preserves all pullbacks in $\Qnd$ of the form
\begin{equation}\vcenter{
\xymatrix{B \times_{\pi_0(B)} X \ar[d]_{\pi_1} \ar[r]^-{\pi_2}  & X \ar[d]^{\phi} \\
B \ar[r]_{\eta_B} & \pi_0(B)
}\tag{3.1}\label{12}}
\end{equation}
where  $\phi \colon X \rightarrow \pi_0 (B)$ is a morphism of $\mathcal F$ lying in the subcategory $\Qnd^*$. 
\end{theorem}
\begin{proof}
Consider the following commutative diagram where:
\begin{itemize}
\item the exterior rectangle is the pullback (\ref{12}), where $\phi \colon X \rightarrow \pi_0 (B)$ is a surjective homomorphism in the subcategory $\Qnd^*$;
\item the universal property of~the unit $\eta_{B \times_{\pi_0(B)} X} \colon B \times_{\pi_0(B)} X \to \pi_0(B \times_{\pi_0(B)} X)$ induces a unique arrow~$\psi \colon \pi_0\left(B \times_{\pi_0(B)}X\right) \to X$ with $\psi \circ \eta_{B \times_{\pi_0(B)} X} = p_2$;
\item the arrow $\gamma \colon B \times_{\pi_0(B)} X \to B \times_{\pi_0(B)} \pi_0(B \times_{\pi_0(B)} X)$ is the one induced by the universal property of the pullback of $\eta_B$ along $\pi_0(p_1)$.
\end{itemize}
\[\xymatrix{B \times_{\pi_0(B)} X \ar[rrr]^{p_2} \ar[dd]_{p_1} \ar@{.>}[dr]^\gamma \ar[drr]^{\eta_{B \times_{\pi_0(B)} X}} & & & X \ar[dd]^{\phi}\\
 & B \times_{\pi_0(B)} \pi_0(B \times_{\pi_0(B)} X) \ar[r]_-{\pi_2} \ar[dl]_{\pi_1} & \pi_0(B \times_{\pi_0(B)} X) \ar[dr]^{\pi_0(p_1)} \ar@{.>}[ru]^{\psi} & \\
B \ar[rrr]_{\eta_B} & & & \pi_0(B)}\]
By Corollary \eqref{QuandleAdj}, we know that the homomorphism $\gamma$ is surjective. The fact that~$\pi_1 \circ \gamma = p_1$ and $\psi \circ \pi_2 \circ \gamma = p_2$ implies that $\gamma$ is also injective. Indeed, this latter property follows from the fact that the pullback projections $p_1$ and $p_2$ are jointly monomorphic, i.e. if $p_i \circ u = p_i \circ v$ (for $i \in \{1,2\}$) then $u = v$. Accordingly, the arrow $\gamma$ is bijective, thus it is an isomorphism.
We can then consider the following diagram
\[\xymatrix@=45pt{\ar @{} [dr] |{(1)}
 B \times_{\pi_0(B)} X \ar[r]^-{\eta_{B \times_{\pi_0(B)} X}} \ar[d]_{p_1} & \ar @{} [dr] |{(2)} \pi_0(B \times_{\pi_0(B)} X) \ar[d]_{\pi_0(p_1)} \ar[r]^-{\psi} & X \ar[d]^{\phi}\\
B \ar[r]_{\eta_B} &\pi_0( B) \ar@{=}[r]_{1_{\pi_0(B)}} & \pi_0(B)}\]
where both the outer rectangle $(1)+(2)$ and the square $(1)$ are pullbacks. Since~$\eta_B$ is a surjective homomorphism it follows that $(2)$ is a pullback (see Proposition~$2.7$ in~\cite{JK}, for instance). This shows that the pullback (\ref{12}) is preserved by the functor~$\pi_0$, as desired.
\end{proof}
\begin{remark}
In categorical Galois theory the pullback preservation property expressed in Theorem \ref{Admissible} is usually called \emph{admissibility} \cite{JK} of the adjunction with respect to the choice of $\mathcal F$ as class of morphisms.
\end{remark}
\begin{remark}
It is not possible to weaken the assumption on~$\phi \colon X \to \pi_0(B)$, which has to be required to be a \emph{surjective} homomorphism. Indeed, as explained in \cite{E}, the functor $\pi_0$ does not preserve pullbacks of the form~(\ref{12}) when $\phi \colon X \rightarrow \pi_0(B)$ is not required to be surjective. In other words the functor $\pi_0$ is not semi-left-exact~\cite{CHK}. \end{remark}
\begin{remark}
One might wonder if, in general, the functor $\pi_0$ preserves pullbacks of surjective homomorphisms along surjective homomorphisms. The answer is negative, as the following counter-example shows: $\pi_0$ does not even preserve kernel pairs of \emph{split} epimorphisms, in general. This shows that the category $\Qnd$ behaves very differently compared to a semi-abelian category (see~\cite{Gran}, Lemma $8.2$).

Let us consider the \emph{involutive} quandle $A$ (this means that $\lhd = \lhd^{-1}$) with four elements $\{ a, b, c, d \}$ defined by the following table
\[\begin{tabular}{|l|l|l|l|l|}
\hline
	$\lhd$ & a & b & c & d\\
\hline
    a & a $\lhd$ a = a & a $\lhd$ b = a & a $\lhd$ c = a & a $\lhd$ d = b\\
\hline
	b & b $\lhd$ a = b & b $\lhd$ b = b & b $\lhd$ c = b & b $\lhd$ d =a\\
\hline
	c & c $\lhd$ a = c & c $\lhd$ b = c & c $\lhd$ c = c & c $\lhd$ d = c\\
\hline
	d & d $\lhd$ a = d & d $\lhd$ b = d & d $\lhd$ c = d & d $\lhd$ d = d\\
\hline
\end{tabular}\]
and the trivial quandle $B$ with two elements $\{ x, y\}$. Let $f \colon A \to B$ be defined by~$f(a) =f(b) =f(c) = x$ and $f(d) = y$. This quandle homomorphism is surjective, and it is even split by the quandle homomorphism $s \colon B \to A$ defined by~$s(x) = c$ and~$s(y) = d$. Its kernel pair $\Eq(f)$ is not preserved by the functor~$\pi_0$. \linebreak Indeed, $[(a,b)]$ and $[(a,a)]$ are distinct elements in $\pi_0(\Eq(f))$ (since $(d,d)$ is the only member of $\Eq(f)$ acting non trivially on $\Eq(f)$), while the corresponding images $([a],[b])$ and $([a],[a])$ are equal in $\Eq(\pi_0(f))$. Accordingly, $\Eq(\pi_0(f))$ is not isomorphic to~$\pi_0(\Eq(f))$.
\end{remark}

Consider the pair $(\mathcal E,\mathcal M)$ of classes of arrows, where
$\mathcal{E}$ is given by the morphisms in $\mathcal F$ inverted by the functor $\pi_0$, and $\mathcal M$ is the class of morphisms $f\colon A \rightarrow B$ in $\mathcal F$ such that the natural square 

\begin{equation}
\vcenter{\xymatrix{A \ar[r]^f \ar[d]_{\eta_A}& B \ar[d]^{\eta_B} \\
\pi_0 (A) \ar[r]_{\pi_0 (f) }& \pi_0 (B)
}\tag{3.2}\label{32}}
\end{equation}
induced by the unit $\eta$ of the adjunction is a pullback. The arrows in the class~$\mathcal M$ are called \emph{trivial extensions} in categorical Galois theory \cite{JK}. The two classes $(\mathcal E,\mathcal M)$ of surjective homomorphisms form a factorisation system for the class $\mathcal F$ in $\Qnd$, as we shall show here below. The morphisms belonging to these two classes can be described as follows:
\begin{proposition}
A surjective homomorphism $f \colon A \rightarrow B$ belongs to the class $\mathcal E$\\  if and only if $\Eq(f) \subset \sim_{\mathsf{Inn}(A)}$. 
\end{proposition}



\begin{proof}
%
%
The fact that $\pi_0$ inverts a surjective homomorphism $f \colon A \rightarrow B$ obviously implies that $\Eq(f) \subset \sim_{\mathsf{Inn}(A)}$.

Conversely, suppose now that $\Eq(f) \subset \sim_{\mathsf{Inn}(A)}$, so that we have the following commutative diagram
\[\xymatrix{ &\sim_{\mathsf{Inn}(A)} \ar@<2pt>[d] \ar@<-2pt>[d] \ar@{.>}[r]^-{\overline{f}} &\sim_{\mathsf{Inn}(B)} \ar@<2pt>[d]^{p_2} \ar@<-2pt>[d]_{p_1}\\
\Eq(f) \ar@{.>}[ur] \ar@<2pt>[r] \ar@<-2pt>[r]& A \ar[r]^f \ar[d]_{\eta_A} & B\ar[d]^{\eta_B}  \ar@{.>}[ld]^\phi \\
& \pi_0(A) \ar[r]_{\pi_0 (f)} & \pi_0(B) }\] 
where the induced dotted homomorphism $\overline{f}$ is a surjective \linebreak homomorphism~($f(\sim_{\mathsf{Inn}(A)}) =\sim_{\mathsf{Inn}(B)}$ by Corollary \eqref{inducedissurjective}),
and the induced dotted homomorphism~$\phi$ is such that $\phi \circ f = \eta_A$.
It follows that $\phi \circ p_1= \phi \circ p_2$, and there exists a unique morphism~$\psi \colon \pi_0(B) \rightarrow \pi_0(A)$ with $\psi \circ \eta_B = \phi$, which is the inverse of $\pi_0 (f)$. 
\end{proof}
\begin{remark}
For a surjective homomorphism $f \colon A \rightarrow B$ the condition~$\Eq(f) \subset  \sim_{\mathsf{Inn}(A)}$ says the following: if $f(a) = f(a')$, then there is an automorphism $\phi \in \mathsf{Inn}(A)$ such that $\phi(a) = a'$. In other words, $f$ can only identify elements of $A$ belonging to the same connected component.
\end{remark}
\begin{proposition}\label{classeM}
A surjective homomorphism $f \colon A \rightarrow B$ belongs to the class $\mathcal M$ \\  if and only if $\Eq(f) \cap \sim_{\mathsf{Inn}(A)} = \Delta_A$.
\end{proposition}
\begin{proof}
This follows immediately from Corollary \eqref{QuandleAdj}. Indeed, given the commutative square \eqref{32}, we know that the factorisation $(\eta_A,f) \colon A \rightarrow \pi_0 (A) \times_{\pi_0 (B) }  B$ is a surjective homomorphism, which will be also injective if and only if its kernel congruence $\Eq(f) \cap \sim_{\mathsf{Inn}(A)} $ is the identity relation $\Delta_A$ on $A$.
\end{proof}
\begin{remark}
A surjective homomorphism $f \colon A \rightarrow B$ belongs to $\mathcal M$ when the following implication holds: $$\forall a, a' \in A, \, ( f(a)=f(a') ) \wedge (\exists \phi \in \mathsf{Inn}(A)\,  \mathrm{with } \, \phi (a)=a') \Rightarrow a=a'.$$
\end{remark}
We now show that any $f \colon A \to B$ in $\mathcal F$ has an $(\mathcal E$, $\mathcal M$) factorisation:

\begin{proposition}\label{existence}
Let $f \colon A \to B$ be a surjective homomorphism in $\Qnd$, then it has a factorisation $\tilde{f} \circ p$, where~$p$ belongs to $\mathcal{E}$ and $\tilde{f}$ belongs to $\mathcal M$.
\end{proposition}

\begin{proof}
Consider the following commutative diagram
\[\xymatrix{\Eq(f) \cap  \sim_{\mathsf{Inn}(A)}  \ar[r] \ar[d] & \sim_{\mathsf{Inn}(A)} \ar@<2pt>[d] \ar@<-2pt>[d] & & \\
\Eq(f) \ar@<2pt>[r] \ar@<-2pt>[r] \ar@{.>}[d] & A \ar[rr]^f \ar[rd]_p & & B\\
\Eq(\tilde{f}) \ar@<2pt>[rr] \ar@<-2pt>[rr] & & \frac{A}{ \Eq(f) \cap \sim_{\mathsf{Inn}(A)}} \ar@{.>}[ru]_{\tilde{f}} &} \]
where $p$ is the canonical quotient, and $\tilde{f}$ is the unique homomorphism such \linebreak that~$\tilde{f} \circ~p = f$: this homomorphism $\tilde{f}$ is defined by $\tilde{f} ([a]) = f(a)$ for \linebreak any $[a] \in~\frac{A}{ \Eq(f) \cap \sim_{\mathsf{Inn}(A)}}$ .
Note that by construction $\Eq(p) \subset \sim_{\mathsf{Inn}(A)}$, so that $p$ is a surjective homomorphism inverted by $\pi_0$. Furthermore, one has the equalities $$\Delta_{\frac{A}{ \Eq(f) \cap \sim_{\mathsf{Inn}(A)}} }= p(\Eq(f) \cap \sim_{\mathsf{Inn}(A)}) = \Eq(\tilde{f}) \cap {\sim}_{\mathsf{Inn} ( \frac{A}{ \Eq(f) \cap \sim_{\mathsf{Inn}(A)}})},$$ 
where $p(\sim_{\mathsf{Inn}(A)} ) =  {\sim}_{\mathsf{Inn} ( \frac{A}{ \Eq(f) \cap \sim_{\mathsf{Inn}(A)}})}$ thanks to Corollary \ref{inducedissurjective}. 
 By Proposition \ref{classeM} it follows that $\tilde{f}$ belongs to $\mathcal M$.
\end{proof}

The classes  \[\mathcal{E} = \{  f \colon A \to B \, \vert \, f \in \mathcal F \,  \mathrm{and } \,  \Eq(f) \subset  \sim_{\mathsf{Inn}(A)} \}  \] 
and 
\[\mathcal{M} = \{f \colon A \rightarrow B \, \vert  \,  f \in \mathcal F \,  \mathrm{and } \, \Eq(f) \cap \sim_{\mathsf{Inn}(A)}  = \Delta_A\}\]
  form a factorisation system for $\mathcal F$ in the category $\Qnd$ of quandles:

\begin{proposition}
 $(\mathcal E,\mathcal M)$ is a factorisation system for $\mathcal F$ in $\Qnd$.
\end{proposition}
\begin{proof}
The condition (i) in Definition \ref{system} is easily checked, while condition (ii) is guaranteed by Proposition \ref{existence}.
To check the condition (iii) in the definition of a factorisation system for $\mathcal F$ consider any
commutative diagram
\begin{equation}
\vcenter{\xymatrix{A \ar[r]^f \ar[d]_g & B \ar[d]^h \\
C \ar[r]_m & D}\tag{3.3}\label{33}}
\end{equation}
in $\Qnd$  where $f\in \mathcal{E}$ and $m \in \mathcal{M}$. We have to show the existence of a unique morphism $t \colon B \rightarrow C$ such that $t \circ f =g$ and $m\circ t = h$.
Consider the cube
\[\xymatrix{A\ar[rr]^f \ar[dr]^{\eta_A} \ar[dd]_g&&B \ar[rd]^{\eta_B} \ar@{-}[d]^<<<<<<<h&\\
&\pi_0(A) \ar[dd]^>>>>>>>{\pi_0(g)} \ar@<2pt>[rr]^<<<<<<{\varphi} \ar@{<-}@<-2pt>[rr]_>>>>>>{\varphi^{-1}} &\ar[d]&\pi_0(B) \ar[dd]^{\pi_0(h)} \\
C \ar@{-}[r]_<<<<<<<<<m \ar[dr]_{\eta_C} &\ar[r]&D \ar[dr]^{\eta_D}&\\
&\pi_0(C) \ar[rr]_{\pi_0(m)}&&\pi_0(D)}\]
where the bottom face is a pullback since $m$ belongs to $\mathcal M$, and~$\pi_0 (f)=~\varphi$ is an isomorphism since $f \in \mathcal E$. The universal property of the pullback and the equality $$\pi_0(m) \circ \pi_0(g) \circ \varphi^{-1} \circ \eta_B = \pi_0(h) \circ \eta_B = \eta_D \circ h$$ induce a unique morphism $t \colon B \to C$ such that, in particular, $m \circ t =h$. The equality $t \circ f = g$ follows from the fact that the morphisms $\eta_C$ and $m$ are jointly monomorphic.
\end{proof}


\section{Comparison with another factorisation system}

Finally, let us compare this factorisation system with another one in $\Qnd_{\rm{RegEpi}}$. In~\cite{BLRY}, E. Bunch, P. Lofgren, A. Rapp and D. N. Yetter showed that every  surjective homomorphism in $\Qnd$ has a canonical factorisation whose second component is what the authors of that article call a \emph{rigid quotient}, namely a surjective homomorphism $h$ such that $\Inn(h)$ is an isomorphism.

\begin{proposition}\label{factY}
Let $f \colon A \rightarrow B$ be a surjective homomorphism in $\Qnd$. Then~$f$ has a factorisation as $f = h \circ g$, where {$g \colon A \rightarrow A/\sim_{\Ker(\Inn(f))}$} \linebreak and $h \colon A/\sim_{\Ker(\Inn(f))} \rightarrow B$ is such that $\Inn(h)$ is an isomorphism.
\end{proposition}

Thanks to the result in \cite{BLRY}, we now show that the classes $$\mathcal{E}_1 = \{f \colon A \rightarrow B \, \vert \, f \in \mathcal F \,  \mathrm{and } \, \, \Eq(f) = \sim_{\Ker(\Inn(f))}\}$$ and $$\mathcal{M}_1 = \{f \colon A \to B \, \vert \, \, f \in \mathcal F \,  \mathrm{and } \,  \Inn(f) \text{ is an isomorphism} \}$$ form a factorisation system for the class $\mathcal F$ of surjective homomorphisms:
\begin{proposition}
$(\mathcal{E}_1, \mathcal{M}_1)$ is a factorisation system for $\mathcal F$ in $\Qnd$.
\end{proposition}
\begin{proof}
The first axiom in the definition of factorisation system is easy to check, while (ii) is precisely Theorem $8.1$ in \cite{BLRY} (recalled as Proposition \ref{factY} here above). To check the validity of property~(iii) consider a commutative square of surjective homomorphisms (\ref{33}) with $f\in {\mathcal E}_1$ and~$m \in {\mathcal M}_1$.
By applying the functor $\Inn$ to this commutative square we get the commutative diagram of surjective group homomorphisms 
\[\xymatrix{&\Ker(\Inn(g)) \ar[d]_-k& \\
\Ker(\Inn(f)) \ar@{.>}[ru]^\iota \ar[r]^-{k'} &\Inn(A) \ar[r]^{\Inn(f)} \ar[d]_{\Inn(g)} & \Inn(B) \ar[d]^{\Inn(h)} \\
&\Inn(C) \ar[r]_{\Inn(m)} & \Inn(D)}\] 
with $\Inn(m)$ an isomorphism. Accordingly, there is an induced \linebreak inclusion $\iota \colon \Ker(\Inn(f)) \hookrightarrow \Ker(\Inn(g))$ between the kernels such that $k \circ \iota = k'$. This induces an inclusion~$\iota' \colon \sim_{\Ker(\Inn(f))} \rightarrow \ \sim_{\Ker(\Inn(g))}$ of the corresponding kernel congruences in $\Qnd$. Using Proposition \ref{factY}, one obtains an~$({\mathcal E}_1,\mathcal{M}_1)$  factorisation~$\tilde{h} \circ \tilde{g}$ of $g$ as in the diagram
\[\xymatrix{& \sim_{\Ker(\Inn(g))} \ar@<2pt>[d] \ar@<-2pt>[d] & &\\
\sim_{\Ker(\Inn(f))} \ar@{^{(}->}[ru] \ar@<2pt>[r] \ar@<-2pt>[r] &A \ar[dr]^-{\tilde{g}} \ar[rr]^f \ar[dd]_g & & B \ar[dd]^h \ar@{.>}[dl]_-\phi\\
& &A/\sim_{\Ker(\Inn(g))} \ar[dl]^-{\tilde{h}} & \\
&C \ar[rr]_m & & D}\]
This induces a homomorphism $\phi \colon B \to A/\sim{\Ker(\Inn(g))}$ such that $\phi \circ f = \tilde{g}$. The arrow $\tilde{h} \circ \phi$ is the desired factorisation showing the orthogonality of ${\mathcal E}_1$ and ${\mathcal M}_1$.
\end{proof}

By comparing this factorisation system with the one considered in the previous section one remarks that $\mathcal{E}_1 \subset \mathcal{E}$, since $\Ker(\Inn(f)) \subset  \Inn(A)$ and, consequently,~$\mathcal{M} \subset \mathcal{M}_1$. 

\begin{remark}
We finally observe that the factorisation system $(\mathcal{E}_1,\mathcal{M}_1)$ does not have the property that $g$ belongs to ${\mathcal E}_1$ whenever
$f \circ g$ and $f$ belong to ~$\mathcal{E}_1$. This shows a difference with the factorisation system  $(\mathcal{E},\mathcal{M})$ of $\mathcal F$ in $\Qnd$ considered in the previous section (the class $\mathcal E$ obviously satisfies this property). \\
Consider the following commutative diagram of involutive quandles: 

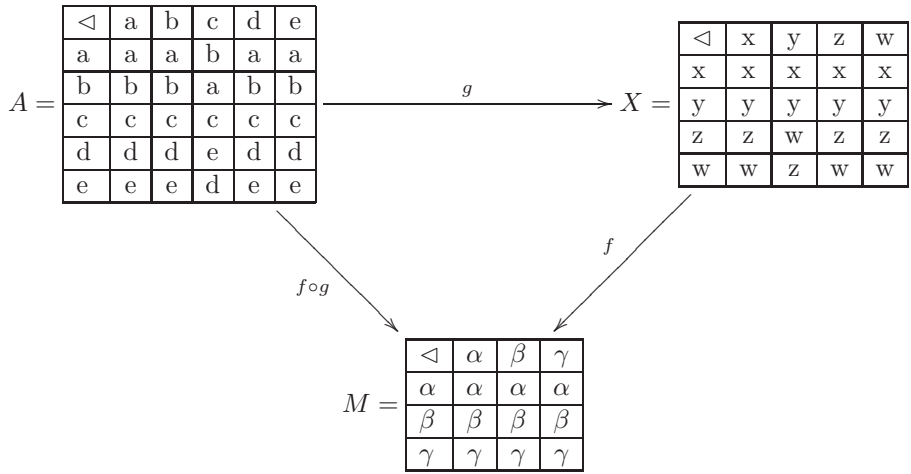
\begin{xy}
(0,0)*+{A=\begin{tabular}{|l|c|c|c|c|c|}
  \hline
  $\lhd$ & a & b & c & d & e\\
  \hline
  a & a & a & b & a & a \\
  \hline
  b & b & b & a & b & b \\
  \hline  
  c & c & c & c & c & c \\
  \hline
  d & d & d & e & d & d \\
\hline
  e & e & e & d & e & e \\
\hline
\end{tabular}}="a"; (80,0)*+{X=\begin{tabular}{|l|c|c|c|c|}
  \hline
  $\lhd$ & x & y & z & w \\
  \hline
  x & x & x & x & x  \\
  \hline
  y & y & y & y & y  \\
  \hline  
  z & z & w & z & z  \\
  \hline
  w & w & z & w & w \\
\hline
\end{tabular} }="x";%
(40,-40)*+{M=\begin{tabular}{|l|c|c|c|}
  \hline
  $\lhd$ & $\alpha$ & $\beta$ & $\gamma$ \\
  \hline
  $\alpha$ & $\alpha$ & $\alpha$ & $\alpha$ \\
  \hline
  $\beta$ & $\beta$ & $\beta$ & $\beta$  \\
  \hline  
  $\gamma$ & $\gamma$ & $\gamma$ & $\gamma$  \\
  \hline
\end{tabular}}="m"

{\ar^-g "a"; "x"}
{\ar_-{f \circ g} "a"; "m"}
{\ar_-f "x"; "m"} 
\end{xy}

Let $g \colon A \to X$ be defined by $g(a)=g(b)=x$, $g(c) = y$, $g(d)=z$ and $g(e) = w$, and let $f \colon X \to M$ be defined by $f(x)=\alpha$, $f(y)=\beta$ and $f(z)=f(w)=\gamma$. One can show that $f= \eta_X$ and $f\circ g = \eta_A$ so that both $f$ and $g\circ f$ are in $\mathcal{E}_1$. To see that~$g$ is not in $\mathcal{E}_1$, remark that $(a,b) \in \Eq(g)$ but the only member of $\Inn(A)$ linking them is $\rho_c$ which does not belong to $\Ker(\Inn(g))$ (since $\Inn(g)(\rho_c) = \rho_y$).

\end{remark}


\end{document}